\documentclass[12pt]{amsart}
\usepackage{amssymb, amsmath, amsfonts, amsthm, esint}
\usepackage[alphabetic]{amsrefs}
\usepackage[margin=0.8 in]{geometry}
\usepackage[colorlinks=true, linkcolor=blue]{hyperref}

\def \eps {\varepsilon}

\DeclareMathOperator{\Ric}{Ric}

\DeclareMathOperator{\Hess}{Hess}

\DeclareMathOperator{\vol}{vol}

\DeclareMathOperator{\diver}{div}

\newcommand{\oo}[1]{\overline{#1}}

\title{Zhong-Yang type eigenvalue estimate with integral curvature condition}
\author{Xavier Ramos Oliv\'e}
\address{Department of Mathematics\\
University of California\\
Riverside, CA 92521}
\email{\href{mailto:xramo002@ucr.edu}{xramo002@ucr.edu}}
\author{Shoo Seto}
\address{Department of Mathematics\\
University of California\\
Irvine, CA 92617}
\email{\href{mailto:shoos@uci.edu}{shoos@uci.edu}}
\author{Guofang Wei}
\address{Department of Mathematics\\
         University of California\\
         Santa Barbara, CA 93106}
     \thanks{G.W. is partially supported by NSF DMS 1811558 }
\email{\href{mailto:wei@math.ucsb.edu}{wei@math.ucsb.edu}}
\author{Qi S. Zhang}
\address{Department of Mathematics\\
University of California\\
Riverside, CA 92521}
\thanks{Q.Z. is partially supported by Simons Foundation grant 282153}
\email{\href{mailto:qizhang@math.ucr.edu}{qizhang@math.ucr.edu}}
\date{}

\keywords{Laplace eigenvalue, Integral Ricci curvature}

\theoremstyle{definition}
\newtheorem{theorem}{Theorem}[section]

\newtheorem{remark}{Remark}[section]

\newtheorem{proposition}[theorem]{Proposition}

\makeatletter
\@addtoreset{equation}{section}
\makeatother

\begin{document}

\begin{abstract}
We prove a sharp Zhong-Yang type eigenvalue lower bound for closed Riemannian manifolds with control on integral Ricci curvature.
\end{abstract}
\maketitle

\section{Introduction}
One trend in Riemannian geometry since the 1950's has been the study of how curvature affects global quantities like the eigenvalues of the Laplacian. On a closed Riemannian manifold $(M^n,g)$, assuming that $\Ric_M \geq (n-1)H$ ($H>0$), Lichnerowicz \cite{Lichnerowicz} proved the lower bound $\lambda_1 (M) \geq Hn$, where $\lambda_1$ is the first nonzero eigenvalue of the Laplace-Beltrami operator in $(M^n,g)$,
\begin{equation*}
\diver(\nabla u) := \Delta u =- \lambda_1 u.
\end{equation*}
Obata \cite{Obata1962} proved the rigidity result that equality holds if and only if $M^n$ is isometric to $\mathbb S_H^n$.  In the case $\Ric_M \geq 0$, one can not prove a positive lower bound without a constraint on the diameter of $M$. Note that the diameter constraint is automatic when $H>0$ by Myers' theorem. If the diameter of $M$ is $D$, Li and Yau \cites{LiYau, li} proved a gradient estimate for the first nontrivial eigenfunction and showed that
$$\lambda_1(M) \geq \frac{\pi^2}{2D^2}.$$
Zhong and Yang \cite{ZhongYang} improved this result, and obtained the optimal estimate
$$\lambda_1(M) \geq \frac{\pi^2}{D^2}.$$
 Hang and Wang \cite{HangWang} proved the rigidity result that
 equality holds if and only if $M$ is isometric to $\mathbb S^1$ with radius $\frac{D}{\pi}$. When $H<0$, improving Li-Yau's estimate \cite{LiYau},  Yang showed the following explicit estimate in \cite{Yang90},
\begin{equation}\label{negative}
\lambda_1(M) \geq \frac{\pi^2}{D^2}e^{-C_n\sqrt{(n-1)|H|}D},
\end{equation}
where $C_n = \max\{\sqrt{n-1},\sqrt{2}\}$. The general optimal lower bound estimate for $\lambda_1(M)$ for all $H$ is proved in \cites{Kroger1992, Chen-Wang1997, Andrews-Clutterbuck2013, Zhang-Wang17} using gradient estimate, probabilistic `coupling method'  and modulus of continuity, respectively, see also \cites{Bakry-Qian2000, Andrews-Ni2012}.  The general lower bound gives a comparison of the first eigenvalue to a one-dimensional model and an explicit lower bound was given by Shi and Zhang \cite{Shi-Zhang} which takes the form
\begin{equation}\label{generallower}
\lambda_1(M) \geq 4(s-s^2)  \frac{\pi^2}{D^2}+ s (n-1)H \ \ \mbox{for all} \  s \in (0,1).
\end{equation}

In recent years, there has been an increasing interest in relaxing the pointwise curvature assumption, by assuming a bound in an $L^p$ sense as in \cite{Gallot1988}.  In the fundamental work \cite{PetersenWei} the basic Laplacian comparison and Bishop-Gromov volume comparison have been extended to integral Ricci curvature. Many topological invariants can be expressed in terms of $L^p$ norms of the curvature, and these bounds are also more suitable than pointwise bounds in the study of Ricci flow.  To be more precise, let $\rho\left( x\right) $ be the smallest eigenvalue for the Ricci tensor. For a constant $H \in \mathbb R$, let $\rho_H$ be the amount of Ricci curvature lying below $(n-1)H$, i.e.
\begin{equation}\label{rhoH}
\rho_H=\max\{-\rho(x)+(n-1)H, 0\}.
\end{equation}
We will be concerned mainly with $\rho_0$, the negative part of $\Ric$. The following quantity measures the amount of Ricci curvature lying below $(n-1)H$ in an $L^p$ sense

$$\bar k(p,H)=\left(\frac{1}{\vol(M)}\int_M \rho_H^p dv\right)^{\frac{1}{p}}= \left(\fint_M \rho_H^p dv\right)^{\frac{1}{p}}.$$
Clearly $\bar k(p,H) = 0$ iff $\Ric_M \ge (n-1)H$.

In \cite{Gallot1988}, Gallot obtained a lower bound for $\lambda_1(M)$ for closed manifolds with diameter bounded from above and $\bar k(p,H)$ small by a heat kernel estimate, see Theorem~\ref{Gallot} below.
The estimate is not optimal though. When  $H>0$, Aubry \cite{Aubry2007} obtained an optimal lower bound estimate for the first nonzero eigenvalue, which recovers the Lichnerowicz estimate.  Recently the second and third authors  \cite{Seto-Wei2017} extended this to the $p$-Laplacian.

In this paper we obtain an optimal estimate for $H=0$, recovering the Zhong-Yang estimate.  Namely,
\begin{theorem} \label{main1}
	Let $(M^n,g)$ be a closed Riemannian manifold with diameter $\le D$ and $\lambda_1(M)$ be the first nonzero eigenvalue.  For any $\alpha \in (0,1)$,  $p > \frac n2$, $n\geq 2$, there exists $\epsilon(n,p,\alpha,  D)>0$ such that if $\bar k (p, 0) \le \epsilon$, then
	\begin{equation*}
	\lambda_1(M) \ge \alpha \frac{\pi^2}{D^2}.
	\end{equation*}
\end{theorem}
The constant $\epsilon$ can be explicitly computed, see \S \ref{appendix}.

The proofs for $\lambda_1(M)$ we mentioned before for  pointwise lower Ricci curvature bound do not work well with integral curvature condition. Here we use a gradient estimate similar to the one of Li and Yau. Their technique can not be applied directly in the integral curvature case, as it relies on a pointwise lower bound to control the term $\Ric (\nabla u, \nabla u)$ coming from the Bochner formula.
To overcome this difficulty, we use the technique developed by Zhang and Zhu in \cite{ZhangZhu1} and \cite{ZhangZhu2}. See also \cite{Carron, Rose2018, Ramos}. The strategy consists in introducing an auxiliary function $J$ via a PDE that absorbs the curvature terms appearing in the Bochner formula, and to find appropriate bounds for $J$ (see \S \ref{estimates}). We also follow the approach of \cite{li}, that uses an ODE comparison technique instead of the original barrier functions of \cite{ZhongYang} (see \S \ref{comparison}).

\begin{remark}
 In general, $p>\frac{n}{2}$ and the smallness of $\bar k(p,0)$ are both necessary conditions, see e.g. \cite{DaiWeiZhang}. In particular, for our estimate, the example of the dumbbells of Calabi shows that only assuming that $\bar k(p,0)$ is bounded is not enough: consider a dumbbell $D_\epsilon \subseteq \mathbb{R}^3$, consisting of two equal spheres joined by a thin cylinder of length $l$ and radius $\epsilon$, with smooth necks. Assume without loss of generality that $\vol (D_\epsilon) \leq 1$. Then, as explained in \cite{Cheeger}, $\lambda_1 (D_\epsilon) \leq C \epsilon^2$, so no positive lower bound is possible, as we can consider a sequence of dumbbells with $\epsilon \rightarrow 0$. Notice that in this example, $\bar k(p,0)$ can not be made small. This follows from Gauss-Bonnet: since $D_\epsilon$ is homeomorphic to $S^2$, the integral of the sectional curvature over $D_\epsilon$ is $4\pi$. However, over the two spheres it is close to $8\pi$, and over the cylinder it's $0$. Hence, the two necks contain negative curvature, that amounts to almost $-4\pi$. This implies that $\vol (D_\epsilon) \bar k(1,0) \approx 4\pi$. We can make the construction so that $\bar k(1,0) >2\pi$, thus for $p>\frac{n}{2} = 1$ we have $\bar k(p,0)\geq \bar k(1,0) >2\pi$, so the integral curvature can not be made small.
\end{remark}

\begin{remark}
 Since in the case $\Ric_M \geq 0$ one has rigidity, a natural question would be to ask if one could get almost rigidity in the Gromov-Hausdorff sense (see \cite{Sakai}), i.e. if $\lambda_1(M)$ is close to $\lambda_1(S^1) = \frac{\pi^2}{D^2}$, can we conclude that $M$ is close to $S^1$ in the Gromov-Hausdorff topology? As was explained in \cite{HangWang} page 8, this is not true, as one could consider the shrinking sequence of boundaries of $\epsilon$-neighborhoods of a line segment with length $\pi$, whose eigenvalues converge to $1$, but that converges in the Gromov-Hausdorff sense to the line segment.
\end{remark}

\begin{remark}
As the proof given in \cite{Yang90} for the case $H<0$ is similar to the case $H=0$, adjusting our proof accordingly should yield an integral curvature version of the estimate \eqref{negative}.  We conjecture that  integral curvature versions of the estimate \eqref{generallower} and the optimal lower bound estimate when $H<0$ should also hold.
\end{remark}

The paper is organized as follows: in \S \ref{estimates} we prove estimates on the auxiliary function $J$ mentioned above, in \S \ref{comparison} we prove the sharp gradient estimate needed to derive our main theorem, and we prove Theorem~\ref{main1} in \S \ref{proofmainthm}. Finally, in \S \ref{appendix} we have an appendix with explicit estimates on $\epsilon$ (the upper bound of $\bar k(p,0)$) depending on the Sobolev and Poincar\'e constants, $p$, $n$ and $D$.

{\bf Acknowledgements.} The authors would  like to thank 
Christian Rose for his interest in the paper and  Hang Chen for helpful comments in an earlier version of the paper. We also thank Jian Hong Wang for carefully checking the earlier version and finding a typo.

\section{Estimates on the auxiliary function}\label{estimates}

In what follows, for $f\in L^p(M)$, we use the notation

\begin{equation*}
\|f\|_p^* := \left(\fint_M |f|^p dv\right)^{\frac{1}{p}}.
\end{equation*}

First we recall an earlier eigenvalue and Sobolev constant estimate for closed manifolds with integral Ricci curvature bounds which we will use.
\begin{theorem} \cite[Theorem 3,6]{Gallot1988}  \label{Gallot}
	Given $M^n$ closed Riemannian manifold with diameter $D$, for $p> n/2$, $H \in \mathbb R$,  there exist $\varepsilon \left(
	n,p, H, D\right) >0, \ C\left( n,p,H, D\right) >0$ such that if $\bar k(p,H) \leq
	\varepsilon,$  then 	
	Is$(M) \le C\left( n,p,H, D\right) /\vol (M)^{\frac 1n}$, where Is$(M) =\sup \{\frac{\vol (\Omega) \}^{1-\frac{1}{n}} }{ \vol (\partial \Omega)} : \Omega \subset M, \vol(\Omega) \le \frac 12 \vol (M) \}.$
In particular,
\begin{equation}\label{roughbound}
 \lambda_1(M) \geq \Lambda_{rough} (n,p,H,D) >0,
\end{equation}
and for any $u\in W^{1,2}(M)$,
\begin{equation}\label{Sobolev}
\|u-\bar{u}\|^*_{\frac{2p}{p-1}} \leq C_s \|\nabla u\|_2^*
\end{equation}
\begin{equation}\label{Sobolev2}
\|u\|^*_{\frac{2p}{p-1}} \leq  C_s \|\nabla u\|_2^* + \|u\|_2^*
\end{equation}
where $\bar{u} = \fint_M u$, the average of $u$, and $C_s = \left(\frac{p}{p-1}\right)^{\frac{1}{2}}C(n,p,H,D)$.
\end{theorem}
In \cite{Gallot1988} the weaker isoperimetric constant was obtained, it was improved to the optimal power above in
\cite{Petersen-Sprouse1998}.

\begin{remark}
 From the estimate on $\lambda_1$ we can derive a Poincar\'e inequality. Notice that

 $$\Lambda_{rough} \leq \lambda_1 = \inf_{\fint_M f dv = 0} \frac{\fint_M |\nabla f|^2dv}{\fint_M f^2 dv} \leq \frac{\fint_M |\nabla w|dv}{\fint_M |w-\overline{w}|^2dv},$$
 where $w\in H^1(M)$. Hence we have

 \begin{equation}\label{Poincare}
  (\|w-\overline{w}\|_2^*)^2 \leq \Lambda_{rough}^{-1}(\|\nabla w\|_2^*)^2.
 \end{equation}
\end{remark}

As mentioned in the introduction, in the proof of the gradient estimate in Proposition \ref{gradientestimate} we introduce an auxiliary function $J$ that absorbs the curvature terms. To be able to derive a sharp lower bound for $\lambda_1(M)$, we need to construct and estimate $J$ from a PDE as follows. For $\tau >1$ and $\sigma \geq 0$, consider
\begin{equation}\label{Jeqn}
\Delta J - \tau\frac{|\nabla J|^2}{J}-2J\rho_0 = -\sigma J.
\end{equation}
Here $\rho_0$ is defined as in \eqref{rhoH}. Using the transformation $J = w^{-\frac{1}{\tau-1}}$, we see that this equation is equivalent to
\begin{equation}\label{weqn}
 \Delta w + Vw = \tilde{\sigma} w,
\end{equation}
where $V = 2(\tau-1)\rho_0$ and $\tilde{\sigma} = (\tau-1)\sigma$. We choose $-\tilde{\sigma}$ be the first eigenvalue of the operator $\Delta +V$; in particular, if $\Ric_M \geq 0$ we have $\sigma = \tilde{\sigma} = 0$, and $J \equiv 1$ is a solution to \eqref{Jeqn}. The main goal of this section is to prove the following propositions.
\begin{proposition}\label{sbounds}
 There exists $\epsilon (n,p,D,\tau)>0$ such that if $\bar k(p,0)\leq \epsilon$, then there is a number $\sigma$ and a corresponding function $J$ solving \eqref{Jeqn} such that
  \begin{equation}
0 \leq \sigma \leq 4\epsilon. \label{sigmabounds}
  \end{equation}
\end{proposition}
\begin{proof}
 Since $-\tilde{\sigma}$ is the first eigenvalue, $w$ doesn't change sign. In particular, by possibly scaling $w$, we can assume that $w\geq 0$, $\|w\|_2^* = 1$. By integrating equation \eqref{weqn} over $M$ we get
 \begin{equation*}
\int_M Vw dv = \tilde{\sigma} \int_M w dv.
 \end{equation*}
 Since $w\geq 0$ and $V\geq 0$, we conclude that $\tilde{\sigma} \geq 0$, so $\sigma \geq 0$.

 To obtain the upper bound, multiply equation \eqref{weqn} by $w$, integrate over $M$, and divide by $\vol (M)$. This way, we obtain:
\begin{equation*}
\tilde{\sigma} = \tilde{\sigma}\fint_M w^2dv = \fint_M w\Delta w dv + \fint_M Vw^2 dv = - \fint_M |\nabla w|^2 dv + \fint_M Vw^2dv.
\end{equation*}
Define the average of $w$
\begin{equation*}
\overline{w} := \fint_M w dv \leq \|w\|_{2}^* = 1.
\end{equation*}
 Then, using the Sobolev inequality \eqref{Sobolev}, for $p>n/2$
 \begin{align*}
  \tilde{\sigma} &=- \fint_M |\nabla w|^2 dv + \fint_M V(w+\overline{w}-\overline{w})^2dv \\
         &\leq -\fint_M |\nabla w|^2 dv + 2\fint_M V(w-\overline{w})^2dv + 2\fint_M V(\overline{w})^2dv\\
         &\leq -\fint_M |\nabla w|^2 dv + 2\|V\|_{p}^* \left( \fint_M (w-\overline{w})^{\frac{2p}{p-1}}dv\right)^{\frac{p-1}{p}} + 2\|V\|_{1}^*\\
         &\leq -\fint_M |\nabla w|^2 dv + 2C^2_s\|V\|_{p}^* \fint_M |\nabla w|^2 dv + 2\|V\|_{1}^*.
 \end{align*}
where $C_s$ is the Sobolev constant. Since $2C^2_s \|V\|_{p}^* \leq 4C^2_s (\tau-1)\epsilon$, choosing $\epsilon>0$ small enough so that $4C^2_s (\tau-1)\epsilon\leq 1$, we deduce
\begin{equation*}
\tilde{\sigma} \leq 2\|V\|_{1}^* \leq 2 \|V\|_{p}^* \leq 4(\tau-1)\epsilon .
\end{equation*}
Hence
\begin{equation*}
 \sigma \leq 4\epsilon.
\end{equation*}
\end{proof}

\begin{proposition}\label{Jbounds}
For any $\delta>0$, there exists $\epsilon (n,p,D,\tau)>0$ and a solution $J$ to \eqref{Jeqn} such that if $\bar k(p,0)\leq \epsilon$ then
\begin{equation}\label{J-1bound}
|J-1|\leq \delta.
\end{equation}
\end{proposition}
\begin{proof}
 \ \\
 {\bf Claim 1:} $\|w-\overline{w}\|_2^* \leq K_1\sqrt{\epsilon}$, where $K_1 = K_1(p,n,D,\tau)$.

 Going back to equation \eqref{weqn}, we have that, since $\overline{w} \leq 1$,
 \begin{align*}
  -\Delta w (w-\overline{w}) &= (V-\tilde{\sigma})w(w-\overline{w})\\
  & = (V-\tilde{\sigma})(w-\overline{w})^2 + \overline{w}(w-\overline{w})(V-\tilde{\sigma})\\
  &\leq |V-\tilde{\sigma}|(w-\overline{w})^2 + \overline{w}((w-\overline{w})^2|V-\tilde{\sigma}| + |V-\tilde{\sigma}|)\\
  &\leq 2(w-\overline{w})^2|V-\tilde{\sigma}| + |V-\tilde{\sigma}|.
 \end{align*}
After integration by parts and dividing by $\vol(M)$:
 \begin{align*}
  \fint_M |\nabla w|^2dv &\leq 2\fint_M |V-\tilde{\sigma}|(w-\overline{w})^2 dv + \fint_M |V-\tilde{\sigma}|dv\\
  & \leq 2\left( \fint_M |V-\tilde{\sigma}|^{p} dv \right)^{1/p} \left( \fint_M (w-\overline{w})^{\frac{2p}{p-1}} dv\right)^{\frac{p-1}{p}} + \|V-\tilde{\sigma}\|_{1}^*\\
  & \leq 2C^2_s \left(\|V\|_{p}^*+\| \tilde{\sigma} \|_{p}^* \right)\fint_M |\nabla w|^2 dv + \|V\|_{1}^*+\| \tilde{\sigma} \|_{1}^*\\
  &\leq 12C^2_s (\tau-1)\epsilon\fint_M |\nabla w|^2 dv + 6(\tau-1)\epsilon.
 \end{align*}
 Then choosing $\epsilon$ small enough so that $12C^2_s (\tau-1)\epsilon\leq \frac{1}{2}$, we deduce
\begin{equation*}
\fint_M |\nabla w|^2 dv \leq 12(\tau-1)\epsilon .
\end{equation*}
 Finally, using Poincar\'e inequality \eqref{Poincare},
\begin{equation*}
(\|w-\overline{w}\|_2^*)^2 = \fint_M |w-\overline{w}|^2 dv \leq \Lambda_{rough}^{-1} \fint_M |\nabla w|^2 dv \leq 12\Lambda_{rough}^{-1}(\tau-1)\epsilon \equiv K_1^2 \epsilon.
\end{equation*}
\noindent{\bf Claim 2: $\|w-\overline{w}\|_\infty \leq K_2(\sqrt{\epsilon}+\epsilon).$}

Denote $h:= w-\overline{w}$. To derive the $L^\infty$ bound for $h$, we will use Moser's iteration on a closed manifold. The technique written below is a slight modification of the one used in \cite{WeiYe}, introducing a potential term, (c.f. \cite{HanLin}). Notice that $h$ satisfies
\begin{equation*}
 -\Delta h = (V-\tilde{\sigma})h+\overline{w}(V-\tilde{\sigma}).
\end{equation*}
Let $g= V-\tilde{\sigma}$ and $f = \overline{w}(V-\tilde{\sigma})$, then $h$ satisfies
\begin{equation}
 -\Delta h = gh+f.
\end{equation}
Thus, $h$ is a weak solution, in the sense that
\begin{equation*}
 \int_M \nabla h \nabla \phi dv = \int_M gh\phi dv+\int_M f\phi dv
\end{equation*}
for all nonnegative $\phi \in W^{1,2}(M)$.

 Define $\overline{h} = h^+ + k$, where $k = \|f\|_{p}^*$. Notice that $\nabla \overline{h} = 0$ if $h\leq0$, and $\nabla h = \nabla\overline{h}$ if $h>0$. Consider $\phi = \frac{\overline{h}^l}{\vol (M)} \in W^{1,2}(M)$. Then for some $l\geq 1$ we have
 \[ l \fint_M |\nabla \overline{h}|^2\overline{h}^{l -1} dv \leq \fint_M gh\overline{h}^{l} dv + \fint_M f\overline{h}^l dv. \]
 Hence
\begin{align}\label{eq1}
   l \fint_M |\nabla \overline{h}|^2\overline{h}^{l -1} dv &\leq \fint_M |g|\overline{h}^{l+1} dv + \fint_M |f|\frac{\overline{h}^{l+1}}{k} dv\\ \nonumber
   &\leq \fint_M \psi \overline{h}^{l+1} dv\\ \nonumber
   &\leq \|\psi\|_{p}^* \left(\|\overline{h}\|_{(l+1)\frac{p}{p-1}}^*\right)^{l+1}, \nonumber
  \end{align}
 where $\psi = |g|+\frac{|f|}{k}$ (in particular $1\leq \|\psi\|_p^* \leq 1 + \|g\|_p^*$). Note that

 \[ \left|\nabla \left(\overline{h}^{\frac{l+1}{2}}\right)\right|^2 = \left|\frac{l+1}{2}\overline{h}^{\frac{l-1}{2}} \nabla \overline{h} \right|^2 = \frac{(l+1)^2}{4}\overline{h}^{l-1}|\nabla \overline{h}|^2.\]
So
 \begin{align*}
  \left(\|\nabla(\overline{h}^{\frac{l+1}{2}})\|_2^* \right)^2 \leq \frac{(l+1)^2}{4l} \|\psi\|_p^* \left( \|\overline{h}\|_{(l+1)\frac{p}{p-1}}^*\right)^{l+1}.
 \end{align*}

 Let $s=\frac{2p}{n}>1$ and $r=\frac{s+1}{2} > 1$.  Note that $\frac{p}{p-1} = \frac{sn}{sn-2}$.  Using the Sobolev inequality \eqref{Sobolev2} for $u\in W^{1,2}(M)$ and $q > \frac{n}{2}$,
 \[\|u\|^*_{\frac{2q}{q-1}} \leq C_s \|\nabla u\|_2^* + \|u\|_2^*,\]
and let $q = \frac{rn}{2}$.  We make this choice, $q>n/2$, so that we can treat the cases $n=2$ and $n>2$ together. Note that if $q=n/2$, then $\frac{2q}{(q-1)}=\frac{2n}{(n-2)}$ is the usual Sobolev exponent. Then we have
 \begin{align}
 \begin{split}\label{hineq}
 \left( \|\overline{h}\|_{(l+1)\frac{rn}{rn-2}}^*\right)^{\frac{l+1}{2}} = \|\overline{h}^{\frac{l+1}{2}}\|_{\frac{2rn}{rn-2}}^* &\leq C_s \|\nabla ( \overline{h}^{\frac{l+1}{2}})\|_2^* + \|\overline{h}^{\frac{l+1}{2}}\|_2^*\\
 &\leq \mathcal{A}_l (\|\overline{h}\|_{(l+1)\frac{sn}{sn-2}}^*)^{\frac{l+1}{2}} +  (\|\overline{h}\|_{l+1}^*)^{\frac{l+1}{2}},
 \end{split}
 \end{align}
where $\mathcal{A}_l = C_s \frac{l+1}{2\sqrt{l}}\sqrt{\|\psi\|_{p}^*}$ from \eqref{eq1}.

Let $a:= \frac{n(s-r)}{sn-2} >0$ so that it satisfies
\begin{align*}
a+(1-a)\left(\frac{rn}{rn-2}\right) = \frac{sn}{sn-2}.
\end{align*}
Then by H\"older and Young's inequality
\begin{align*}
xy \leq \xi x^\gamma + \xi^{-\frac{\gamma^*}{\gamma}}y^{\gamma^*},
\end{align*}
with $\gamma = \left((1-a)\left(\frac{rn}{rn-2}\right)\left(\frac{sn-2}{sn}\right)\right)^{-1}$ and $\gamma^* = \left(\frac{sn-2}{sn}a\right)^{-1}$,
\begin{align*}
(\|\oo h\|_{(l+1)\frac{sn}{sn-2}}^*)^{\frac{l+1}{2}} &=\left(\fint_M\oo h^{(l+1)\frac{sn}{sn-2}} dv\right)^{\frac{sn-2}{2sn}}\\
&\leq \left(\fint_M\oo h^{(l+1)\frac{rn}{rn-2}}dv \right)^{(1-a)\frac{sn-2}{2sn}}\left( \fint_M\oo h^{l+1}dv \right)^{\frac{sn-2}{2sn}a}\\
&\leq \xi\left(\fint_M\oo h^{(l+1)\frac{rn}{rn-2}}dv\right)^{\frac{rn-2}{2rn}} + \xi^{-\frac{(1-a)}{a}\frac{rn}{rn-2}}\left(\fint_M\oo h^{l+1}dv \right)^{\frac{1}{2}}.
\end{align*}
Inserting this into inequality \eqref{hineq} and setting $\xi = \frac{1}{2}\mathcal{A}_l^{-1}$, we have
\begin{align*}
\left(\|\oo h\|_{(l+1)\frac{rn}{rn-2}}^*\right)^{\frac{l+1}{2}} &\leq \mathcal{A}_l\left(\xi\left(\fint_M \oo h^{(l+1)\frac{rn}{rn-2}} dv\right)^{\frac{rn-2}{2rn}} + \xi^{-\frac{(1-a)}{a}\frac{rn}{rn-2}}\left(\fint_M\oo h^{l+1} dv\right)^{\frac{1}{2}} \right) + \left(\|\oo h\|_{l+1}^*\right)^{\frac{l+1}{2}}\\
&= \frac{1}{2}\left(\|\oo h\|_{(l+1)\frac{rn}{rn-2}}^*\right)^{\frac{l+1}{2}} + \left(2^{\frac{1-a}{a}\frac{rn}{rn-2}}\mathcal{A}_l^{\frac{1-a}{a}\frac{rn}{rn-2}+1}+1\right)\left(\|\oo h\|_{l+1}^*\right)^{\frac{l+1}{2}}.
\end{align*}
Renaming $l+1$ by $l > 1$, and $\mu = \frac{rn}{rn-2}$, we obtain
\begin{equation*}
\|\oo h\|^*_{l\mu} \leq \left((2\mathcal{A}_{l-1})^{\frac{s}{s-r}}+2\right)^{\frac{2}{l}}\|\oo h\|_{l}^*.
\end{equation*}
Let $l_k = l\mu^k$.  Then
\begin{align*}
\|\oo h\|_{l_k}^* \leq \prod_{j=1}^{k}\left((2\mathcal{A}_{l_{j-1}-1})^{\frac{s}{s-r}}+2\right)^{\frac{2}{l_j}}\|\oo h\|_l^*.
\end{align*}
Note that $\mathcal{A}_l = O(\sqrt{l})$ so that $\mathcal{A} = \lim_{k \to \infty}\prod_{j=1}^{k}\left((2\mathcal{A}_{l_{j-1}-1})^{\frac{s}{s-r}}+2\right)^{\frac{2}{l_j}} < \infty$ hence
\begin{equation}\label{Moserineq}
\|\oo h\|_\infty \leq \mathcal{A} \|\oo h\|^*_l
\end{equation}
for $l > 1$.  In particular, let $l=2$ so that
\begin{align*}
\sup h^+ \leq \|\overline{h}\|_\infty &\leq \mathcal{A}\|\overline{h}\|_{2}^* \\
& = \mathcal{A} \|h^+ + \|f\|^*_p\, \|_2^*\\
& \leq \mathcal{A}(\|h\|_2^* + \|\overline{w}(V-\tilde{\sigma})\|_p^*)\\
& \leq \mathcal{A}(K_1\sqrt{\epsilon} + 6(\tau -1)\epsilon)\\
&\leq \mathcal{A}(K_1 + 6(\tau -1))(\sqrt{\epsilon}+\epsilon) \equiv K_2 (\sqrt{\epsilon}+\epsilon).
\end{align*}
Here we have used Claim 1.  Since $-h$ satisfies the same equation (except for a sign in $f$), we conclude that
 \begin{equation*}
|w-\overline{w}|=|h|\leq K_2 (\sqrt{\epsilon}+\epsilon).
\end{equation*}

 \noindent{\bf Claim 3:} For $\epsilon$ small enough, $\overline{w}>\frac{1}{2}$.

From the previous claim, we know that $w\leq \overline{w}+K_2(\sqrt{\epsilon}+\epsilon)$. Since $\|w\|_2^* = 1$ and $\overline{w}\leq 1$,
\begin{equation*}
 1= \fint_M w^2 dv \leq \fint_M (\overline{w}+K_2(\sqrt{\epsilon}+ \epsilon))^2 dv = (\overline{w}+K_2(\sqrt{\epsilon}+ \epsilon))^2 \leq \overline{w}^2 +2K_2(\sqrt{\epsilon}+\epsilon)+ K_2^2(\sqrt{\epsilon} + \epsilon)^2.
\end{equation*}
Hence, choosing $\epsilon>0$ small enough
\begin{equation*}
 \frac{1}{4} < 1-2K_2(\sqrt{\epsilon}+\epsilon)-K_2^2(\sqrt{\epsilon}+\epsilon)^2 \leq \overline{w}^2,
\end{equation*}
so that $\overline{w}> \frac{1}{2}$.  This allows us to finish the proof of the lemma. Consider the function $w_2:= w/\overline{w}$. $w_2$ satisfies the same equation as $w$, and we know by claims 2 and 3 that
\begin{equation*}
1-\tilde{\delta}  \leq w_2 \leq  1+\tilde{\delta},
\end{equation*}
where $\tilde{\delta} := 2K_2(\sqrt{\epsilon}+\epsilon)$. Define $J:= w_2^{-\frac{1}{\tau-1}}$. We can establish the bounds for $J$ using the 1st order Taylor polynomial of $f(x) = x^{-\frac{1}{\tau-1}}$ near $x=1$, on the domain $(1-\tilde{\delta}, 1+\tilde{\delta})$. We know that $f(x) = 1+ R_1(x)$, where the remainder can be estimated by
\begin{equation*}
    |R_1| = |f'(x^*)(x-1)| \leq \frac{2\tilde{\delta}}{(\tau-1)(1-\tilde{\delta})^{\frac{\tau}{\tau-1}}},
  \end{equation*}
where $x^*\in  (1-\tilde{\delta}, 1+\tilde{\delta})$. Choosing $\epsilon>0$ small enough so that $\frac{2\tilde{\delta}}{(\tau-1)(1-\tilde{\delta})^{\frac{\tau}{\tau-1}}}\leq \delta$, we get the estimate
\begin{equation*}
  |J-1| \leq \delta,
\end{equation*}
concluding the proof of the proposition.
\end{proof}

\section{Sharp gradient estimate}\label{comparison}
As in the $\Ric_M \geq 0$ case, to obtain a sharp estimate we need to consider the following ODE with an additional parameter $\eta \in \mathbb{R}$,
\begin{equation}\label{ODE}
\begin{cases}
(1-u^2)Z''(u) +  uZ'(u) = -\eta u & \text{ on } [-1,1]\\
Z(0) = 0\\
Z(\pm 1) = 0,
\end{cases}
\end{equation}
which has the explicit solution
\begin{align*}
Z(u) = \frac{2\eta}{\pi}\left(\arcsin(u) + u\sqrt{1-u^2}\right)  -\eta u.
\end{align*}
Furthermore, the function $Z(u)$ satisfies the following inequalities.
\begin{proposition}\label{ODEprop}
For numbers $\eta = 1+\delta$ and $J\leq \eta$, we have
\begin{equation}\label{ODEineq1}
\eta^{-1}(Z')^2 - 2J^{-1}Z''Z+Z' \geq 0,
\end{equation}
\begin{equation}\label{ODEineq2}
2Z-uZ'+1 \geq 1 -\eta, \quad u \in [-1,1],
\end{equation}
\begin{equation}\label{ODEineq3}
\eta(1-u^2) \geq 2|Z|.
\end{equation}
\end{proposition}

\begin{proof}
By direct computation,
\begin{align*}
Z'(u) &= -\eta + \frac{4\eta}{\pi}\sqrt{1-u^2},\\
Z''(u) &= -\frac{4\eta}{\pi}\frac{u}{\sqrt{1-u^2}}.
\end{align*}
We first show the inequality \eqref{ODEineq1}. It was shown in \cite{li} that the solution to \eqref{ODE} when $\eta =1$, namely, $z(u) = \frac{2}{\pi}(\arcsin(u)+u\sqrt{1-u^2})-u$ satisfies
\begin{align*}
(z')^2 - 2z'' z + z' \geq 0.
\end{align*}
For $\eta > 0$, the solution is simply given by rescaling $Z = \eta z$.  Note that $z,u$, and $-z''$ shares sign, so that $-zz'' \geq 0$ on $[-1,1]$.   Since $J\leq \eta$,
\begin{align*}
\eta^{-1}(Z')^2 - 2J^{-1}Z''Z+Z' &= \eta(z')^2-2J^{-1}\eta^2z''z+\eta z'\geq \eta((z')^2 - 2z'' z + z') \geq 0.
\end{align*}
Similarly, for \eqref{ODEineq3}, the case $\eta = 1$ was shown in \cite{li} and we obtain the result by rescaling.  Direct computation yields \eqref{ODEineq2}.
\end{proof}
We mention that in Proposition \ref{ODEprop}, $u$, $J$ are generic parameters, and $u$ is a variable.  In order to prove the main result, we need to make the following choice.

Let $\phi$ be a nontrivial eigenfunction corresponding to $\lambda_1$ and normalized so that for $0\leq a <1$, $a+1=\sup \phi$ and $a-1 = \inf \phi$.  Set $u:=\phi-a$, so that $\Delta u = -\lambda_1(u+ a)$.
\begin{proposition}\label{gradientestimate}
Let $\delta > 0$ be fixed.   Let $J$ be the solution \eqref{Jeqn} with $\tau = \frac{3+4\delta}{2\delta}$. Suppose $M$ is a closed Riemannian manifold with $\bar k(p,0) \leq \epsilon$ so that Proposition \ref{sbounds} and \ref{Jbounds} holds for this choice of $\delta$.  Then $u$ defined as above satisfies the gradient estimate
\begin{equation}\label{gradientestimate}
J|\nabla u|^2 \leq \tilde{\lambda}(1-u^2)+2a\lambda_1 Z(u),
\end{equation}
where $\tilde{\lambda} := C_1\lambda_1+C_2$ with
\begin{equation}\label{constant1}
C_1 := \frac{1+\delta +\sqrt{A}}{1-\sqrt{B}}
\end{equation}
and
\begin{equation}\label{constant2}
 C_2 := \frac{\sigma}{2(1-\sqrt{B})}\left(\frac{\tilde{Z}}{\sqrt{A}}+\frac{1}{2\sqrt{B}}\right),
\end{equation}
where $A = A(\delta) := 2\delta(1+\delta)$, $\displaystyle B = B(\delta) := \frac{\delta(5+\delta)}{(1-\delta)}$, and $Z(u)$ is defined as in \eqref{ODE} for $\eta = 1+\delta$ with $\displaystyle \sup_{[-1,1]} Z = \tilde{Z}\leq (0.116)\eta$, and $\sigma$ is given in Proposition \ref{sbounds}.
\end{proposition}
\begin{remark}\label{constantremark}
Note that when $\delta \to 0$, then $C_1 \to 1$.  Also, when $\sigma/\sqrt{\delta} \to 0$, then $C_2 \to 0$.
\end{remark}

\begin{proof}
Consider $\xi < 1$, and denote, for simplicity, $u := \xi (\phi-a)$. Note that comparing with the previously defined $u$, there is an extra factor $\xi$ so that $-\xi \leq u \leq \xi$. Then this $u$ satisfies
\begin{equation}\label{xiueq}
\Delta u = -\lambda_1(u+\xi a).
\end{equation}
Consider for a constant $c$,
\begin{align*}
Q := J|\nabla u|^2-c(1-u^2)-2a\lambda_1 Z(u).
\end{align*}
To obtain \eqref{gradientestimate}, we want to show that $Q\leq 0$.  By \eqref{ODEineq3} and compactness of $M$, we can choose a suitable $c$ so that $Q=0$ at max point.  For the rest of the proof, we fix such a constant $c$ so that $\max Q=0$. 

Our goals is to show that $c\leq C_1\lambda_1+C_2$ for some constants $C_1,C_2$ such that $(C_1,C_2) \to (1,0)$ as $\delta\to 0$, $\sigma/\sqrt{\delta} \to 0$.  Taking $\xi \to 1$ will give us the gradient estimate.

If $c\leq (1+\delta)\lambda_1$, then we are done so we can assume that
\begin{equation}\label{cdeltaineq}
c > (1+\delta)\lambda_1,
\end{equation}
and in particular, $c > J\lambda_1$.  Let the maximum point of $Q$ be $x_0$.  Then $|\nabla u(x_0)| > 0$, since if $\nabla u(x_0) = 0$, then
\begin{align*}
0 &= -c(1-u^2(x_0)) - 2a\lambda_1 Z(u(x_0))\\
&\leq -c(1-\xi^2)+a\lambda_1 \eta(1-\xi^2) \\
&=(a\lambda_1 \eta-c)(1-\xi^2)<0,
\end{align*}
a contradiction.

For convenience, we write $Z = Z(u)$. By direct computation,
\begin{align*}
\nabla Q &= (\nabla J)|\nabla u|^2+J\nabla(|\nabla u|^2) +2cu\nabla u -2a\lambda_1 Z'\nabla u,
\end{align*}
which at the maximum point is
\begin{equation}\label{1storderQ}
\nabla(|\nabla u|^2) = -\frac{\nabla J}{J}|\nabla u|^2-2J^{-1}cu\nabla u + 2J^{-1}a\lambda_1 Z'\nabla u.
\end{equation}
Let $e_1 = \frac{\nabla u}{|\nabla u|}$ and complete to an orthonormal basis $\{e_j\}_{j=1}^n$.  Then at the maximal point we have
\begin{align*}
\Hess u(\nabla u,\nabla u)= -\frac{1}{2J}\langle \nabla J,\nabla u\rangle |\nabla u|^2-c\frac{u}{J}|\nabla u|^2 + \frac{a\lambda_1 Z'}{J}|\nabla u|^2
\end{align*}
so that
\begin{equation}\label{hess1}
\Hess u(e_1,e_1) = -\frac{1}{2J}\langle \nabla J,\nabla u\rangle - c\frac{u}{J}+\frac{a\lambda_1 Z'}{J},
\end{equation}
and for $j\neq 1$,
\begin{equation}\label{hessj}
\Hess u(e_1,e_j) = -\frac{1}{2J}\langle \nabla J,e_j\rangle|\nabla u|.
\end{equation}
Again by direct computation,
\begin{align*}
\Delta Q &= (\Delta J)|\nabla u|^2 + 2\langle \nabla J,\nabla (|\nabla u|^2)\rangle+J\Delta|\nabla u|^2\\
&\hspace{0.2 in}+2c|\nabla u|^2+2cu\Delta u -2a\lambda_1 Z''|\nabla u|^2 - 2a\lambda_1 Z'\Delta u.
\end{align*}
By Bochner formula and \eqref{xiueq},
\begin{align*}
\Delta Q &= (\Delta J)|\nabla u|^2 + 2\langle \nabla J,\nabla(|\nabla u|^2)\rangle \\
&\hspace{0.2 in}+ J\left(2|\Hess u|^2 -2\lambda_1|\nabla u|^2 +2\Ric(\nabla u,\nabla u)\right) \\
&\hspace{0.2 in}+ 2c|\nabla u|^2-2cu\lambda_1 (u+\xi a)-2a\lambda_1 Z''|\nabla u|^2+2a\lambda_1^2 Z'(u+\xi a).
\end{align*}
At the maximal point of $Q$, using $\Ric\geq -\rho_0$ and \eqref{1storderQ}, we have
\begin{align*}
0 &\geq (\Delta J)|\nabla u|^2 -2\frac{|\nabla J|^2}{J}|\nabla u|^2-4c\frac{u}{J}\langle \nabla J,\nabla u\rangle+4\frac{a\lambda_1 Z'}{J}\langle \nabla J,\nabla u\rangle \\
&\hspace{0.2 in} +2J|\Hess u|^2 - 2\lambda_1 J|\nabla u|^2-2J\rho_0|\nabla u|^2\\
&\hspace{0.2 in} + 2c|\nabla u|^2-2cu\lambda_1 (u+\xi a)-2a\lambda_1 Z''|\nabla u|^2+2a\lambda_1^2 Z'(u+\xi a).
\end{align*}
Using \eqref{hess1} and \eqref{hessj} we also have the lower bound of the Hessian term
\begin{align*}
|\Hess u|^2 &\geq \sum_{j=1}^n (\Hess u(e_1,e_j))^2\\
&=\frac{1}{4J^2}|\nabla J|^2|\nabla u|^2 +c\frac{u}{J^2}\langle \nabla J,\nabla u\rangle - \frac{a\lambda_1 Z'}{J^2}\langle \nabla J,\nabla u\rangle+c^2\frac{u^2}{J^2}-\frac{2cua\lambda_1 Z'}{J^2}+\frac{a^2\lambda_1^2(Z')^2}{J^2}.\\
\end{align*}
Inserting the Hessian lower bound we have
\begin{align}
\begin{split}\label{ineqwithoutcauchy}
0 &\geq \left((\Delta J) -\frac{3}{2}\frac{|\nabla J|^2}{J}-2J\rho_0\right)|\nabla u|^2-2c\frac{u}{J}\langle \nabla J,\nabla u\rangle+2\frac{a\lambda_1 Z'}{J}\langle \nabla J,\nabla u\rangle \\
&\hspace{0.2 in}   +\frac{2c^2u^2}{J}-\frac{4cua\lambda_1 Z'}{J}+\frac{2a^2\lambda_1^2(Z')^2}{J} - 2\lambda_1 J|\nabla u|^2\\
&\hspace{0.2 in} + 2c|\nabla u|^2-2cu\lambda_1 (u+\xi a)-2a\lambda_1 Z''|\nabla u|^2+2a\lambda_1^2 Z'(u+\xi a).
\end{split}
\end{align}
Let $\displaystyle \beta = \frac{2\delta}{1+\delta}$.  By Cauchy-Schwarz, we bound the mixed term as follows,
\begin{align*}
2\frac{a\lambda_1 Z'}{\sqrt{J}}\left\langle \frac{\nabla J}{\sqrt{J}},\nabla u\right\rangle \geq -\beta\frac{a^2\lambda_1^2 (Z')^2}{J}-\frac{|\nabla J|^2}{\beta J}|\nabla u|^2
\end{align*}
and
\begin{align*}
-2c\frac{u}{\sqrt{J}}\left\langle \frac{\nabla J}{\sqrt{J}},\nabla u\right\rangle \geq -\delta c^2\frac{u^2}{J}-\frac{|\nabla J|^2}{\delta J}|\nabla u|^2.
\end{align*}
Plugging these into \eqref{ineqwithoutcauchy} we deduce
\begin{align*}
0 &\geq \left((\Delta J) -\left(\frac{3+4\delta}{2\delta}\right)\frac{|\nabla J|^2}{J}-2J\rho_0\right)|\nabla u|^2\\
&\hspace{0.2 in}+(2-\delta)\frac{c^2u^2}{J} -\frac{4cua\lambda_1 Z'}{J} +(2-\beta)\frac{a^2\lambda_1^2(Z')^2}{J} - 2\lambda_1 J|\nabla u|^2\\
&\hspace{0.2 in} + 2c|\nabla u|^2-2cu\lambda_1 (u+\xi a)-2a\lambda_1 Z''|\nabla u|^2+2a\lambda_1^2 Z'(u+\xi a).
\end{align*}
Now using the fact that $Q=0$ at the maximal point, written explicitly
\begin{equation}\label{maxQeq}
|\nabla u|^2 = cJ^{-1}(1-u^2)+2a\lambda_1J^{-1}Z,
\end{equation}
we substitute to the second and third lines of the above so that
\begin{align*}
0 &\geq \left((\Delta J) -\left(\frac{3+4\delta}{2\delta}\right)\frac{|\nabla J|^2}{J}-2J\rho_0\right)|\nabla u|^2\\
&\hspace{0.2 in} -\frac{4cua\lambda_1 Z'}{J} - 2\lambda_1 (c(1-u^2)+2a\lambda_1 Z)\\
&\hspace{0.2 in} + 2cJ^{-1}(c(1-u^2)+2a\lambda_1 Z)-2cu\lambda_1 (u+\xi a)\\
&\hspace{0.2 in}-2a\lambda_1 Z''J^{-1}(c(1-u^2)+2a\lambda_1 Z)+2a\lambda_1^2 Z'(u+\xi a) \\
&\hspace{0.2 in}+(2-\delta)\frac{c^2u^2}{J} +(2-\beta)\frac{a^2\lambda_1^2(Z')^2}{J}.
\end{align*}
Using the equation \eqref{Jeqn} with $\tau = \frac{3+4\delta}{2\delta}$, substituting \eqref{maxQeq} again, and noting that $2-\beta = 2\eta^{-1}$,
\begin{align*}
0&\geq -\sigma(c(1-u^2)+2a\lambda_1 Z)\\
&+2a^2\lambda_1^2(\xi Z'-2J^{-1}Z''Z+\eta^{-1}J^{-1}(Z')^2))\\
&-2ac\lambda_1J^{-1}((1-u^2)Z''+uZ'+\xi uJ)\\
&+2a\lambda(cJ^{-1}-\lambda_1)(2Z-uZ')\\
&-2\lambda_1c+2c^2J^{-1}-2ac\lambda_1J^{-1}+2a\lambda_1^2\\
&-\delta J^{-1}c^2u^2+2ac\lambda_1J^{-1}-2a\lambda_1^2.
\end{align*}
After some re-arranging we have
\begin{align*}
0 &\geq 2a^2\lambda_1^2(\xi Z'-2J^{-1} Z''Z + \eta^{-1}(Z')^2)+2a^2\lambda_1^2\eta^{-1}(J^{-1}-1)(Z')^2\\
&\hspace{0.2 in}-2ac\lambda_1 J^{-1}((1-u^2)Z''+uZ'+\xi uJ)\\
&\hspace{0.2 in}+ 2a\lambda_1(cJ^{-1}-\lambda_1)(2Z-uZ' +1)\\
&\hspace{0.2 in}+2(cJ^{-1}-\lambda_1)(c-a\lambda_1) \\
&\hspace{0.2 in}+(c\sigma -c^2J^{-1}\delta)u^2-c\sigma -2a\sigma\lambda_1 Z.
\end{align*}
The first line is grouped by terms with $a^2\lambda_1^2$, then using \eqref{cdeltaineq}, we group terms as products of $cJ^{-1}-\lambda_1$ since it has a sign, the second line grouped from the remaining terms with a factor $ac\lambda_1$, and the last are remaining terms which would mostly be zero for the pointwise bounded case since $\sigma = 0$.  Using the ODE \eqref{ODE}, the inequalities \eqref{ODEineq1},\eqref{ODEineq2}, and \eqref{cdeltaineq}, we have
\begin{align*}
0 &\geq 2a^2\lambda_1^2 (\xi-1)Z'+2a^2\lambda_1^2\eta^{-1}(J^{-1}-1)(Z')^2\\
&\hspace{0.2 in}-2ac\lambda_1 J^{-1}(\xi J-\eta)u\\
&\hspace{0.2 in}+2a\lambda_1 (cJ^{-1}-\lambda_1)(1-\eta)\\
&\hspace{0.2 in}+2(cJ^{-1}-\lambda_1)(c-a\lambda_1)\\
&\hspace{0.2 in}-c\sigma - 2a\sigma \lambda_1 Z - c^2J^{-1}\delta.
\end{align*}
Since $\eta > 1$ and $u \geq -\xi\geq -1$, $-\eta \leq Z'\leq \eta(\frac{4}{\pi}-1)\leq \eta$, and using Proposition~\ref{Jbounds} we replace $J$ by either $1-\delta$ or $1+\delta$ appropriately, and noting that $2a^2\lambda^2\eta^{-1}(J^{-1}-1)(Z')^2 \geq -2\delta c^2$ we get
\begin{align*}
0 &\geq -2a^2\lambda_1^2\eta(1 - \xi) \\
&\hspace{0.2 in} -2ac\lambda_1 \left( \eta -\xi+\frac{\delta \eta}{1-\delta}\right)\\
&\hspace{0.2 in} -2a\lambda_1\left(c-\lambda_1+ \frac{c\delta }{1-\delta}\right)(\eta-1)\\
&\hspace{0.2 in} +2\left(c-\lambda_1 - \delta c\right)(c-a\lambda_1) \\
&\hspace{0.2 in}-c\sigma -2a\sigma \lambda_1 Z - \frac{c^2\delta}{1-\delta} -2\delta c^2.
\end{align*}
Recall that $\tilde{Z} := \sup Z = O(\eta)$. Our goal now is to obtain an quadratic inequality in terms of $(c-\lambda_1)$.  Using $\eta \geq 1$, we first rewrite the first term and split the remaining terms so that
\begin{align*}
0&\geq -2a^2\lambda_1^2\eta(\eta-\xi)\\
&-2ac\lambda_1(\eta-\xi)-2ac\lambda_1\frac{\delta \eta}{1-\delta}\\
&-2a\lambda_1(c-\lambda_1)(\eta-1)-2a\lambda_1\frac{c\delta}{1-\delta}(\eta-1)\\
&+2(c-\lambda_1)(c-a\lambda_1)-2\delta c(c-a\lambda_1)\\
&-c\sigma-2a\sigma\lambda_1\tilde{Z}-\frac{c^2\delta}{1-\delta}-2\delta c^2.
\end{align*}
 Since $0 \leq a < 1$ and noting that $\xi \leq 1$, $\eta \geq 1$, and $c \geq \lambda_1$, we replace $a$ by 1 or 0 according to the sign so that
\begin{align*}
0&\geq -2\lambda_1^2\eta(\eta-\xi)\\
&-2c\lambda_1(\eta-\xi)-2c\lambda_1\frac{\delta \eta}{1-\delta}\\
&-2\lambda_1(c-\lambda_1)(\eta-1)-2\lambda_1\frac{c\delta}{1-\delta}(\eta-1)\\
&+2(c-\lambda_1)^2-4\delta c^2\\
&-c\sigma-2\sigma\lambda_1\tilde{Z}-\frac{c^2\delta}{1-\delta}.
\end{align*}
Combining the first two terms, and adding and subtracting $\lambda_1$, we get
\begin{align*}
0&\geq -2\lambda_1(\eta-\xi)(c-\lambda+\lambda\eta+\lambda)-2c\lambda_1\frac{\delta\eta}{1-\delta}\\
&-2\lambda_1(c-\lambda_1)(\eta-1)-2\lambda_1\frac{c\delta}{1-\delta}(\eta-1) \\
&+2(c-\lambda_1)^2-4\delta c^2-c\sigma-2\sigma\lambda_1\tilde{Z}-\frac{c^2\delta}{1-\delta}.
\end{align*}
After further rearranging we have,
\begin{align*}
0 &\geq -2\lambda_1^2(\eta - \xi)(\eta + 1)+2\left(c-\lambda_1\right)^2 - 4\delta c^2 -c\sigma -2\sigma \lambda_1 \tilde{Z} -\frac{c^2\delta}{1-\delta}\\
&\hspace{0.2 in} -2\lambda_1\left(c-\lambda_1\right)(2\eta-\xi-1) -2\lambda_1 \left( \frac{c\delta }{1-\delta}\right)(2\eta-1).
\end{align*}
Completing the square in terms of $(c-\lambda_1)$, we get
\begin{align*}
\left( (c-\lambda_1) -\frac{\lambda_1}{2}(2\eta-\xi-1)\right)^2 &\leq \lambda_1^2\frac{4(\eta-\eps)(\eta+1)+(2\eta-\xi-1)^2}{4}+\sigma\lambda_1 \tilde{Z}\\
&\hspace{0.2 in}+c\lambda_1\left( \frac{\delta}{1-\delta}\right)(2\eta -1)+c^2\left(\frac{\delta(4-\delta)}{(1-\delta)}\right)+c\frac{\sigma}{2}.
\end{align*}
Using $\lambda_1 \leq c$ we get
\begin{align*}
\left( (c-\lambda_1) -\frac{\lambda_1}{2}(2\eta-\xi-1)\right)^2 &\leq \lambda_1^2\frac{4(\eta-\xi)(\eta+1)+(2\eta-\xi-1)^2}{4}+\sigma\lambda_1 \tilde{Z}\\
&\hspace{0.2 in}+c^2\left(\left( \frac{\delta}{1-\delta}\right)(2\eta -1)+\frac{\delta(4-\delta)}{(1-\delta)}\right)+c\frac{\sigma}{2}.
\end{align*}
Let
\begin{align*}
A &= A(\eta,\xi) :=\frac{4(\eta-\xi)(\eta+1)+(2\eta-\xi-1)^2}{4}, \\
B &= B(\delta) := \left(\frac{\delta(5+\delta)}{(1-\delta)}\right).
\end{align*}
Then
\begin{align*}
\left( (c-\lambda_1) -\frac{\lambda_1}{2}(2\eta-\xi-1)\right)^2 &\leq A\left(\lambda_1 + \frac{\sigma \tilde{Z}}{2A}\right)^2 - \frac{\sigma^2\tilde{Z}^2}{4A} + B\left(c+\frac{\sigma}{4B}\right)^2-\frac{\sigma^2}{16B} \\
&\leq A\left(\lambda_1 + \frac{\sigma \tilde{Z}}{2A}\right)^2 + B\left(c+\frac{\sigma}{4B}\right)^2.
\end{align*}
Using the inequality $\sqrt{a^2+b^2} \leq a +b$, for $a,b \geq 0$,
\begin{align*}
c-\lambda_1 \leq \frac{\lambda_1}{2}(2\eta-\xi-1) + \sqrt{A}\left(\lambda_1 + \frac{\sigma\tilde{Z}}{2A}\right) + \sqrt{B}\left(c+\frac{\sigma}{4B}\right).
\end{align*}
Letting $\xi \to 1$, we have
\begin{align*}
 c\leq \frac{\lambda_1}{1-\sqrt{B}}\left(\eta + \sqrt{A}+\frac{\sigma\tilde{Z}}{2\lambda_1\sqrt{A}}\right) + \left(\frac{\sigma}{4\sqrt{B}(1-\sqrt{B})}\right)
\end{align*}
recalling that $\eta=1+\delta$ so that $A := A(1+\delta,1) = 2\delta(1+\delta)$.
\end{proof}

\section{Proof of Theorem \ref{main1}}\label{proofmainthm}
We now give the sharp eigenvalue lower bound.  By Proposition \ref{gradientestimate} we have
\begin{align*}
J|\nabla u|^2 \leq (C_1\lambda_1 +C_2)(1-u^2)+2a\lambda_1 Z(u).
\end{align*}
By \eqref{roughbound}, we have a rough lower bound of the first eigenvalue, $\lambda_1 \geq \Lambda_{\text{rough}} > 0$ so that
\begin{align*}
\lambda_1 \geq \frac{J|\nabla u|^2}{(C_1+C_2\Lambda_{\text{rough}}^{-1})(1-u^2)+2aZ(u)}.
\end{align*}
Let $b := C_1+C_2\Lambda_{\text{rough}}^{-1}$ and let $\gamma$ be the shortest geodesic connecting the minimum and maximum point of $u$ with length at most $D$.  Recalling that $\max u = 1$ and $\min u=-1$ from the construction given above Proposition \ref{gradientestimate}, integrating the gradient estimate along the geodesic and using change of variables $x(s) = u(\gamma(s))$ and that $Z$ is odd,
\begin{align*}
D\sqrt{\lambda_1} &\geq \sqrt{\lambda_1}\int_{\gamma}ds\\
&\geq (1-\delta)\int_{\gamma} \frac{|\nabla u|ds}{\sqrt{b(1-u^2)+2aZ(u)}}\\
&=(1-\delta)\int_{-1}^1 \frac{dx}{\sqrt{b(1-x^2)+2aZ(x)}}\\
&=(1-\delta)\int_0^1\left(\frac{1}{\sqrt{b(1-x^2)+2aZ(x)}}+\frac{1}{\sqrt{b(1-x^2)-2aZ(x)}}\right)dx\\
&= \frac{1-\delta}{\sqrt{b}}\int_0^1\frac{1}{\sqrt{1-x^2}}\left(\frac{1}{\sqrt{1 + \frac{2aZ(x)}{b(1-x^2)}}} + \frac{1}{\sqrt{1 - \frac{2aZ(x)}{b(1-x^2)}}}\right)dx \\
&\geq \frac{1-\delta}{\sqrt{b}}\int_0^1\frac{1}{\sqrt{1-x^2}} \left(2 +\frac{3a^2(Z(x))^2}{b^2(1-x^2)^2}\right)dx\\
&\geq \frac{1-\delta}{\sqrt{b}}\pi,
\end{align*}
so that
\begin{align*}
\lambda_1 \geq \frac{(1-\delta)^2}{b}\frac{\pi^2}{D^2}.
\end{align*}
Recall in the limiting case (c.f. Remark \ref{constantremark}) that $b\to 1$ as $\delta\to 0$ and $\sigma/\sqrt{\delta} \to 0$.   Let $\displaystyle \alpha := \frac{(1-\delta)^2}{b}$ to obtain the result.

\section{Appendix: estimate of $\epsilon$}\label{appendix}

 In this appendix we will give explicit bounds for $\epsilon$ depending on $p$, $n$, $D$, and in terms of the Sobolev $C_s$ and Poincar\'e $\Lambda_{rough}^{-1}$ constants, which have explicit expressions in \cite{Gallot1988}. Note that the Sobolev and Poincar\'e constants can be estimated and do not change for $\epsilon$ smaller than some fixed number.   We show that it suffices to choose
\begin{equation}\label{eps0}
\epsilon < \min \Bigg\{ (n-1) \left( \frac{1}{BD}\ln \left(1 + \frac{1}{2^{p+1}} \right) \right)^2, \frac{\delta}{12C_s^2(3+2\delta)}, \left( \frac{\sqrt{7}-2}{K_2} \right)^2,\frac{1}{8K_2\left( \frac{4}{3+2\delta}\right)^{\frac{9+6\delta}{3+2\delta}}}\Bigg\}
\end{equation}
 where  $B(p,n) = \left(\frac{2p-1}{p}\right)^{\frac{1}{2}}(n-1)^{1-\frac{1}{2p}}\left( \frac{2p-2}{2p-n} \right)^{\frac{p-1}{2p}}$, $K_2 = \mathcal{A}(K_1+\frac{9+6\delta}{\delta})$, $K_1 = \sqrt{6\Lambda_{rough}^{-1} \left(2+\frac{3}{\delta} \right)}$ and $\mathcal{A}(n,p,D)$ is the constant that appears from Moser's iteration \eqref{Moserineq}.

 It suffices to choose $\epsilon$ smaller than the worst of the following conditions:

\begin{enumerate}
 \item To apply the Sobolev inequalities that follow from Theorem \ref{Gallot}, $\epsilon$ needs to be small enough so that the theorem holds. Using our notation, the condition that needs to be satisfied in \cite[Theorem 3, 6]{Gallot1988} is
 \begin{equation}\label{condition}
\fint_M \left( \frac{\rho_0}{(n-1)\tilde{\alpha}^2}-1 \right)_+^p dv \leq \frac{1}{2}\left(e^{B\tilde{\alpha} D} -1 \right)^{-1}
 \end{equation}
 for some $\tilde{\alpha}>0$.  Multiplying \eqref{condition} by $\tilde{\alpha}^{2p}$ and raising both sides to the power $1/p$ we get
 $$\left( \fint_M \left(\frac{\rho_0}{n-1}-\tilde{\alpha}^2 \right)_+^{p} dv \right)^{\frac{1}{p}} \leq \frac{\tilde{\alpha}^2}{2^{\frac{1}{p}}} (e^{B\tilde{\alpha} D}-1)^{-\frac{1}{p}}.$$
 Note that
 $$\left( \fint_M \left( \frac{\rho_0}{n-1} -\tilde{\alpha}^2\right)_+^p dv \right)^{\frac{1}{p}} \leq \left\|\frac{\rho_0}{n-1}-\tilde{\alpha}^2\right\|_p^* \leq \frac{\bar k(p,0)}{n-1} + \tilde{\alpha}^2 \leq \frac{\epsilon}{n-1}+\tilde{\alpha}^2.$$
 Thus, it suffices to impose that for some fixed $\tilde{\alpha}>0$ we have
 $$\frac{\epsilon}{n-1} + \tilde{\alpha}^2 \leq \frac{\tilde{\alpha}^2}{2^{\frac{1}{p}}} (e^{B\tilde{\alpha} D}-1)^{-\frac{1}{p}},$$
 or equivalently
 $$\epsilon \leq (n-1)\tilde{\alpha}^2 \left(\frac{1}{2^\frac{1}{p} (e^{B\tilde{\alpha} D}-1)^{\frac{1}{p}}} -1 \right).$$
 So choosing $\tilde{\alpha} = \frac{1}{BD}\ln \left( 1+\frac{1}{2^{p+1}} \right)$ we obtain
\begin{equation}\label{eps1}
\epsilon \leq (n-1) \left( \frac{1}{BD}\ln \left(1 + \frac{1}{2^{p+1}} \right) \right)^2
\end{equation}
 \item In Proposition \ref{sbounds} and in Claim 1 of Proposition \ref{Jbounds} we need $\epsilon$ to satisfy
 \begin{equation}\label{eps2}
 \epsilon \leq \frac{1}{24 C_s^2(\tau -1)},
 \end{equation}
where $C_s$ is the Sobolev constant of \eqref{Sobolev}. From the proof of Proposition \ref{gradientestimate}, $\tau = \frac{3+4\delta}{2\delta}$ so we get
$$\epsilon < \frac{\delta}{12C_s^2(3+2\delta)}.$$

 \item In Claim 3 of Proposition \ref{Jbounds} we need $\epsilon$ to satisfy
 $$\frac{1}{4} < 1-2K_2(\sqrt{\epsilon}+\epsilon)-K_2^2(\sqrt{\epsilon }+\epsilon)^2.$$
This implies that
$$\epsilon < \left( -\frac{1}{2} + \sqrt{\frac{1}{4}+\frac{\sqrt{7}-2}{2K_2}} \right)^2.$$
To get a cleaner estimate, notice that since $\mathcal{A}>1$, then $K_2>6$, thus $\frac{\sqrt{7}-2}{2K_2}<\frac{\sqrt{7}-2}{12} < \frac{3}{4}$. Using that for $0<x<3$ we have that $-1+\sqrt{1 + x} > x$, letting $x = \frac{2(\sqrt{7}-2)}{K_2}$, we get
\begin{equation}\label{eps3}
\epsilon < \left( \frac{\sqrt{7}-2}{K_2} \right)^2.
\end{equation}

 \item In Claim 3 of Proposition \ref{Jbounds} we also need $\epsilon$ to be small enough so that
 $$\frac{2\tilde{\delta}}{(\tau -1)(1-\tilde{\delta})^{\frac{\tau}{\tau -1}}} \leq \delta,$$
where $\tilde{\delta} = 2K_2(\sqrt{\epsilon}+\epsilon)$. This condition is equivalent to
$$\left(  \frac{1}{\tilde{\delta}} -1  \right) \tilde{\delta}^{\frac{1}{\tau}} \geq \left( \frac{2}{\delta (\tau -1)} \right)^{\frac{\tau -1}{\tau}} = \left( \frac{4}{3+2\delta} \right)^{\frac{3+2\delta}{3+4\delta}} \equiv C_3(\delta).$$
Note that since $\delta <1$, we have that $\tau >\frac{3}{2}$. Then we get that $\tilde{\delta}^\frac{1}{\tau} > \tilde{\delta}^{\frac{2}{3}}$. It suffices to choose $\epsilon$ small enough so that
$$\left(  \frac{1}{\tilde{\delta}} -1  \right) \tilde{\delta}^{\frac{2}{3}} \geq C_3.$$
Notice that if $\delta< \frac{1}{2}$, then $C_3 >1$. Thus, choosing $\tilde{\delta}< \frac{1}{8C_3^3}$ we get that
$$\left(  \frac{1}{\tilde{\delta}} -1  \right) \tilde{\delta}^{\frac{2}{3}} > (2C_3)^3-\frac{1}{(2C_3)^2} > C_3+7C_3^3-\frac{1}{4C_3^2} >C_3+7-\frac{1}{4}>C_3,$$
so the condition is satisfied. This gives us the last condition
$$\epsilon <-\frac{1}{2}+\sqrt{\frac{1}{4}+\frac{1}{16K_2 \left( \frac{4}{3+2\delta}\right)^{\frac{9+6\delta}{3+2\delta}}}}.$$
As above, to get a cleaner estimate, notice that if $\delta <\frac{1}{2}$ then $$\frac{1}{16K_2 \left( \frac{4}{3+2\delta}\right)^{\frac{9+6\delta}{3+2\delta}}} < \frac{1}{16K_2} < \frac{1}{96} <\frac{3}{4},$$ so using the same property as in (3) we get that it suffices to assume
\begin{equation}\label{eps4}
\epsilon < \frac{1}{8K_2\left( \frac{4}{3+2\delta}\right)^{\frac{9+6\delta}{3+2\delta}}}.
\end{equation}
From \eqref{eps1},\eqref{eps2},\eqref{eps3},and \eqref{eps4}, we arrive at \eqref{eps0}.

\end{enumerate}


\begin{bibdiv}
\begin{biblist}
	\bib{Andrews-Clutterbuck2013}{article}{
		author={Andrews, Ben},
		author={Clutterbuck, Julie},
		title={Sharp modulus of continuity for parabolic equations on manifolds
			and lower bounds for the first eigenvalue},
		journal={Anal. PDE},
		volume={6},
		date={2013},
		number={5},
		pages={1013--1024},
		issn={2157-5045},
		review={\MR{3125548}},
		review={Zbl 1282.35099},
	}

	\bib{Andrews-Ni2012}{article}{
	author={Andrews, Ben},
	author={Ni, Lei},
	title={Eigenvalue comparison on Bakry-Emery manifolds},
	journal={Comm. Partial Differential Equations},
	volume={37},
	date={2012},
	number={11},
	pages={2081--2092},
	issn={0360-5302},
	review={\MR{3005536}},
	review={Zbl 1258.35153},
}

		\bib{Aubry2007}{article}{
		author={Aubry, Erwann},
		title={Finiteness of $\pi_1$ and geometric inequalities in almost
			positive Ricci curvature},
		language={English, with English and French summaries},
		journal={Ann. Sci. \'Ecole Norm. Sup. (4)},
		volume={40},
		date={2007},
		number={4},
		pages={675--695},
		issn={0012-9593},
		review={\MR{2191529}},
	}

	\bib{Bakry-Qian2000}{article}{
	author={Bakry, Dominique},
	author={Qian, Zhongmin},
	title={Some new results on eigenvectors via dimension, diameter, and
		Ricci curvature},
	journal={Adv. Math.},
	volume={155},
	date={2000},
	number={1},
	pages={98--153},
	issn={0001-8708},
	review={\MR{1789850}},
	review={Zbl 0980.58020},
}

	\bib{Blacker-Seto}{article}{
	author={Blacker, Casey},
	author={Seto, Shoo},
	title={First eigenvalue of the p-Laplacian on Kahler manifolds},
	journal={to appear in Proceedings of the AMS},
	eprint={arXiv:1804.10876},
}

\bib{Carron}{article}{
  title={Geometric inequalities for manifolds with Ricci curvature in the Kato class},
  author={Carron, Gilles},
  eprint={arXiv:1612.03027},
}

\bib{Cheeger}{article}{
   author={Cheeger, Jeff},
   title={A lower bound for the smallest eigenvalue of the Laplacian},
   conference={
      title={Problems in analysis},
      address={Papers dedicated to Salomon Bochner},
      date={1969},
   },
   book={
      publisher={Princeton Univ. Press, Princeton, N. J.},
   },
   date={1970},
   pages={195--199},
   review={\MR{0402831}},
}
	\bib{Chen-Wang1997}{article}{
	author={Chen, Mufa},
	author={Wang, Fengyu},
	title={General formula for lower bound of the first eigenvalue on
		Riemannian manifolds},
	journal={Sci. China Ser. A},
	volume={40},
	date={1997},
	number={4},
	pages={384--394},
	issn={1006-9283},
	review={\MR{1450586}},
	review={Zbl 0895.58056},
}

\bib{DaiWeiZhang}{article}{
   author={Dai, Xianzhe},
   author={Wei, Guofang},
   author={Zhang, Zhenlei},
   title={Local Sobolev constant estimate for integral Ricci curvature
   bounds},
   journal={Adv. Math.},
   volume={325},
   date={2018},
   pages={1--33},
   issn={0001-8708},
   review={\MR{3742584}},
   doi={10.1016/j.aim.2017.11.024},
}

\bib{Gallot1988}{article}{
	author={Gallot, Sylvestre},
	title={Isoperimetric inequalities based on integral norms of Ricci
	curvature},
	note={Colloque Paul L\'evy sur les Processus Stochastiques (Palaiseau,
	1987)},
	journal={Ast\'erisque},
	number={157-158},
	date={1988},
	pages={191--216},
	issn={0303-1179},
	review={\MR{976219}},
}

\bib{HanLin}{book}{
   author={Han, Qing},
   author={Lin, Fanghua},
   title={Elliptic Partial Differential Equations},
   series={Courant Lecture Notes},
   volume={1},
   publisher={American Mathematical Society},
   date={2011},
   isbn={978-0-8218-5313-9},
   pages={x+147},
   review={\MR{1669352}},
}

\bib{HangWang}{article}{
   author={Hang, Fengbo},
   author={Wang, Xiaodong},
   title={A remark on Zhong-Yang's eigenvalue estimate},
   journal={Int. Math. Res. Not. IMRN},
   date={2007},
   number={18},
   pages={Art. ID rnm064, 9},
   issn={1073-7928},
   review={\MR{2358887}},
   doi={10.1093/imrn/rnm064},
}
\bib{Kroger1992}{article}{
	author={Kr{\"o}ger, Pawel},
	title={On the spectral gap for compact manifolds},
	journal={J. Differential Geom.},
	volume={36},
	date={1992},
	number={2},
	pages={315--330},
	issn={0022-040X},
	review={\MR{1180385}},
	review={Zbl 0738.58048},
}

\bib{li}{book}{
   author={Li, Peter},
   title={Geometric analysis},
   series={Cambridge Studies in Advanced Mathematics},
   volume={134},
   publisher={Cambridge University Press, Cambridge},
   date={2012},
   pages={x+406},
   isbn={978-1-107-02064-1},
   review={\MR{2962229}},
   doi={10.1017/CBO9781139105798},
}

\bib{LiYau}{article}{
   author={Li, Peter},
   author={Yau, Shing Tung},
   title={Estimates of eigenvalues of a compact Riemannian manifold},
   conference={
      title={Geometry of the Laplace operator},
      address={Proc. Sympos. Pure Math., Univ. Hawaii, Honolulu, Hawaii},
      date={1979},
   },
   book={
      series={Proc. Sympos. Pure Math., XXXVI},
      publisher={Amer. Math. Soc., Providence, R.I.},
   },
   date={1980},
   pages={205--239},
   review={\MR{573435}},
}

\bib{Lichnerowicz}{book}{
   author={Lichnerowicz, Andr\'{e}},
   title={G\'{e}om\'{e}trie des groupes de transformations},
   language={French},
   publisher={Travaux et Recherches Math\'{e}matiques, III. Dunod, Paris},
   date={1958},
   pages={ix+193},
   review={\MR{0124009}},
}

\bib{Obata1962}{article}{
	author={Obata, Morio},
	title={Certain conditions for a Riemannian manifold to be isometric with a sphere},
	journal={J. Math. Soc. Japan},
	volume={14},
	date={1962},
	pages={333--340},
	issn={0025-5645},
}

\bib{Petersen-Sprouse1998}{article}{
   author={Petersen, Peter},
   author={Sprouse, Chadwick},
   title={Integral curvature bounds, distance estimates and applications},
   journal={J. Differential Geom.},
   volume={50},
   date={1998},
   number={2},
   pages={269--298},
   issn={0022-040X},
   review={\MR{1684981}},
}

\bib{PetersenWei}{article}{
   author={Petersen, P.},
   author={Wei, G.},
   title={Relative volume comparison with integral curvature bounds},
   journal={Geom. Funct. Anal.},
   volume={7},
   date={1997},
   number={6},
   pages={1031--1045},
   issn={1016-443X},
   review={\MR{1487753}},
   doi={10.1007/s000390050036},
}

\bib{Ramos}{article}{
   author={Ramos Oliv\'{e}, Xavier},
   title={Neumann Li-Yau gradient estimate under integral Ricci curvature
   bounds},
   journal={Proc. Amer. Math. Soc.},
   volume={147},
   date={2019},
   number={1},
   pages={411--426},
   issn={0002-9939},
   review={\MR{3876759}},
   doi={10.1090/proc/14213},
}

\bib{Rose2018}{article}{
	author={Rose, Christian},
	title={Li–Yau gradient estimate for compact manifolds with negative part of Ricci curvature in the Kato class},
	journal={Ann Glob Anal Geom},
	date={2018},
	eprint={ https://doi.org/10.1007/s10455-018-9634-0},
}


\bib{Sakai}{article}{
   author={Sakai, Takashi},
   title={Curvature---up through the twentieth century, and into the future?
   [translation of S\={u}gaku {\bf 54} (2002), no. 3, 292--307; MR1929898]},
   note={Sugaku Expositions},
   journal={Sugaku Expositions},
   volume={18},
   date={2005},
   number={2},
   pages={165--187},
   issn={0898-9583},
   review={\MR{2182883}},
}
		
		
\bib{Seto-Wei2017}{article}{
	author={Seto, Shoo},
	author={Wei, Guofang},
	title={First eigenvalue of the $p$-Laplacian under integral curvature
		condition},
	journal={Nonlinear Anal.},
	volume={163},
	date={2017},
	pages={60--70},
	issn={0362-546X},
}

\bib{Shi-Zhang}{article}{
   author={Shi, Yu Min},
   author={Zhang, Hui Chun},
   title={Lower bounds for the first eigenvalue on compact manifolds},
   language={Chinese, with English and Chinese summaries},
   journal={Chinese Ann. Math. Ser. A},
   volume={28},
   date={2007},
   number={6},
   pages={863--866},
   issn={1000-8314},
   review={\MR{2396231}},
}

\bib{Urakawa1987}{article}{
	author={Urakawa, Hajime},
	title={Stability of harmonic maps and eigenvalues of the Laplacian},
	journal={Trans. Amer. Math. Soc.},
	volume={301},
	date={1987},
	number={2},
	pages={557--589},
	issn={0002-9947},
	review={\MR{882704}},
	doi={10.2307/2000659},
}

\bib{WeiYe}{article}{
   author={Wei, Guofang},
   author={Ye, Rugang},
   title={A Neumann type maximum principle for the Laplace operator on
   compact Riemannian manifolds},
   journal={J. Geom. Anal.},
   volume={19},
   date={2009},
   number={3},
   pages={719--736},
   issn={1050-6926},
   review={\MR{2496575}},
   doi={10.1007/s12220-009-9080-0},
}

\bib{Yang90}{article}{
	author={Yang, Hong Cang},
	title={Estimates of the first eigenvalue for a compact Riemann manifold},
	journal={Sci. China Ser. A},
	volume={33},
	date={1990},
	number={1},
	pages={39--51},
	issn={1001-6511},
	review={\MR{1055558}},
}

	\bib{Zhang-Wang17}{article}{
	author={Zhang, Yuntao},
	author={Wang, Kui},
	title = {An alternative proof of lower bounds for the first eigenvalue on manifolds},
	journal={Math. Nachr.},
	VOLUME = {290},
	YEAR = {2017},
	NUMBER = {16},
	PAGES = {2708--2713},
	ISSN = {0025-584X},
	review={\MR{3722506}},
	review={Zbl 1379.35206},
}

		

\bib{ZhangZhu1}{article}{
   author={Zhang, Qi S.},
   author={Zhu, Meng},
   title={Li-Yau gradient bound for collapsing manifolds under integral
   curvature condition},
   journal={Proc. Amer. Math. Soc.},
   volume={145},
   date={2017},
   number={7},
   pages={3117--3126},
   issn={0002-9939},
   review={\MR{3637958}},
   doi={10.1090/proc/13418},
}
		
\bib{ZhangZhu2}{article}{
   author={Zhang, Qi S.},
   author={Zhu, Meng},
   title={Li-Yau gradient bounds on compact manifolds under nearly optimal
   curvature conditions},
   journal={J. Funct. Anal.},
   volume={275},
   date={2018},
   number={2},
   pages={478--515},
   issn={0022-1236},
   review={\MR{3802491}},
   doi={10.1016/j.jfa.2018.02.001},
}

\bib{ZhongYang}{article}{
   author={Zhong, Jia Qing},
   author={Yang, Hong Cang},
   title={On the estimate of the first eigenvalue of a compact Riemannian
   manifold},
   journal={Sci. Sinica Ser. A},
   volume={27},
   date={1984},
   number={12},
   pages={1265--1273},
   issn={0253-5831},
   review={\MR{794292}},
}

		

\end{biblist}
\end{bibdiv}
\end{document}